\newif\ifsubsections
\subsectionstrue 

\documentclass[11pt]{amsart}
\usepackage[sc]{mathpazo}          
\usepackage{eulervm}               
\usepackage[scaled=0.86]{berasans} 
\usepackage[scaled=1]{inconsolata} 
\usepackage[T1]{fontenc}

\usepackage[%
	protrusion=true,
	expansion=false,
	auto=false
	]{microtype}

\usepackage{xcolor}
\usepackage{graphicx}
\graphicspath{{./figures/}}

	\definecolor{linkred}{rgb}{0.0,0.0,0.0}
	\definecolor{linkblue}{rgb}{0,0.0,0.0}

\PassOptionsToPackage{hyphens}{url} 
\usepackage[
    setpagesize=false,
    pagebackref,
pdfpagelabels=false,
   pdfstartview={FitH 1000},
    bookmarksnumbered=false,
    linktoc=all,
    colorlinks=true,
    anchorcolor=black,
    menucolor=black,
    runcolor=black,
    filecolor=black,
   linkcolor=linkblue,
	citecolor=linkblue,
	urlcolor=linkred,
]{hyperref}
\usepackage[backrefs,msc-links,nobysame]{amsrefs}

\theoremstyle{plain}
\newtheorem{eq}{Eq}[subsection]
\newtheorem{theorem}[eq]{Theorem}

\newtheorem{lemma}[eq]{Lemma}
\newtheorem{corollary}[eq]{Corollary}
\newtheorem{prop}[eq]{Proposition}

\newtheorem*{theorem*}{Theorem}{\bf}{\it}
\newtheorem*{proposition*}{Proposition}{\bf}{\it}
\newtheorem*{observation*}{Observation}{\bf}{\it}
\newtheorem*{lemmaA*}{Lemma A}{\bf}{\it}
\newtheorem*{conj}{Conjecture}{\bf}{\it}

\theoremstyle{definition}

\newtheorem{example}[eq]{Example}

\theoremstyle{remark}
\newtheorem{remark}[eq]{Remark}
\newtheorem{exercise}[eq]{Exercise}
\long\def\replace#1{#1}

\replace{%

}
\newcommand{\h}{\mathcal H}
\newcommand{\Z}{\mathbb Z}
\newcommand{\R}{\mathbb R}

\newcommand{\D}{\mathbb D}
\newcommand{\C}{\mathbb C}
\newcommand{\T}{\mathbb T}
\newcommand{\dv}{{\rm{div}}}
\newcommand{\dist}{{\rm{dist}}}

\newcommand{\n}{\mathcal{N}}
\newcommand{\p}{\mathcal{P}}
\newcommand{\Q}{\mathcal{Q}}
\newcommand{\Subset}{\subset\subset}
\newcommand{\osc}{\operatorname{osc}}
\newcommand{\grad}{\operatorname{grad}}
\ProvideTextCommandDefault{\cprime}{\tprime}
\begin{document}

%
%
%
%
%
%

\title[Quantitative unique continuation]{Lecture notes on quantitative unique continuation for solutions of second order elliptic equations}

%
%
\author[A. Logunov]{Alexander Logunov}
\address{A.L.: Department of Mathematics, Princeton University, Princeton, NJ, 08544;}
\email{log239@yandex.ru}
\thanks{This work was
completed during the time A.L. served as a Clay Research Fellow.}

\author[E. Malinnikova]{Eugenia Malinnikova}
\address{E.M.: School of Mathematics,  Institute for Advanced Study, 1 Einstein Dr, Princeton, NJ 08540}

  \address{ Department of Mathematical Sciences,
Norwegian University of Science and Technology
7491, Trondheim, Norway}
\email{eugenia.malinnikova@ntnu.no}
\thanks{E.M. is supported  by  Project 275113
of the Research Council of Norway and NSF grant no. DMS-1638352}

%
%

\begin{abstract} 
In these  lectures we present some useful techniques to study quantitative properties of solutions of elliptic PDEs. Our aim is to  outline a proof of a recent result on  propagation of smallness. The ideas are also useful in the study of the zero sets of eigenfunctions of Laplace--Beltrami operator and we  discuss the connection. Some basic facts about second order elliptic PDEs in divergent form are collected in the Appendix at the end of the notes.
\end{abstract}
\maketitle
\thispagestyle{empty}

%
%


\section{Eigenfunctions of Laplace--Beltrami operators}\label{s:ef}
\subsection{Definition}
Let $M$ be an oriented  Riemannian manifold with metric tensor $g=(g_{ij})$, we denote by $|g|$ the absolute value of the determinant of the matrix $g_{ij}$ and by $g^{-1}=(g^{ij})$ the inverse tensor. The Laplace--Beltrami operator on functions on $M$ is defined as the divergence of the gradient. In local coordinates we get
\[\Delta_M f=\frac{1}{\sqrt{|g|}}\dv(\sqrt{|g|}g^{-1}\nabla f),\]
where $\nabla f=(\partial_1 f,..., \partial_n f)$ in choosen coordinates.

Using the metric, one defines the volume form $dV_M$, in local coordinates it becomes $dV_M=\sqrt{|g|}dx_1\wedge...\wedge dx^n$. Futher the gradient of a $C^1$ function $f$ on $M$ is a vector field on $M$ locally given by
\[\grad_M f=\sum_{i,j}(g^{ij}\partial_j f)\partial_i.\] 
The following Green formula holds  for functions $f,h\in W_0^{1,2}(M)$
\[\int_M h\Delta_M f dV_M=-\int_M\langle\grad_M f,\grad_M h\rangle_gdV_M.\]

Let $M$ be a compact manifold  without boundary and consider eigenfunctions $\phi_\lambda$ of the Laplace--Beltrami operator, such that $\Delta_M \phi_\lambda+\lambda \phi_\lambda=0.$
Then 
\[\int_M|\grad_M \phi_\lambda|_g^2dV_M=\lambda\int_M|\phi_\lambda|^2dV_M.\]
All eigenvalues of $-\Delta_M$ are real and non-negative, eigenfunctions corresponding to distinct eigenvalues are orthogonal  since
\[\lambda \int_M \phi_\lambda\phi_\mu dV_M=-\int_M(\Delta_M\phi_\lambda)\phi_\mu dV_M=\mu\int_M \phi_\lambda\phi_\mu dV_M.\]
The eigenvalues form an increasing sequence that tends to infinity,
\[0=\lambda_1<\lambda_2\le\lambda_3\le...\le\lambda_n\le... .\]
The first eigenfunction $\phi_0$ is a constant.
There is an orthonormal basis of eigenfunctions for $L^2(M)$. We refer the reader to \cite[Chapter 1]{Ch} for details.

\begin{example} (Dirichlet Laplacian for a domain in $\R^d$.) Instead of a compact manifold we may also consider a bounded domain $\Omega$ in $\R^d$ and the Laplace operator with the Dirichlet boundary condition
\[\Delta \phi+\lambda \phi=0,\quad \phi|_{\partial\Omega}=0.\]
The first eigenvalue is given by the variational formula
\[\lambda_1(\Omega)=\min_{\phi}\int_{\Omega}|\nabla \phi|^2,\]
where the minimum is taken over all functions $\phi\in W^{1,2}_0(\Omega)$ such that $\int_\Omega|\phi|^2=1$. The formula implies that if $\Omega_1\subset\Omega_2$ then
\[\lambda_1(\Omega_1)\ge \lambda_1(\Omega_2).\]
The first eigenfunction does not change sign and can be choosen positive in $\Omega$; all other eigenfunctions are orthogonal to the first one and thus change sign in $\Omega$. Eigenvalues can be determined by the min-max formula
\[
\lambda_k(\Omega)=\min_{A_k}\max_{\phi\in A_k}\frac{\int_{\Omega}|\nabla \phi|^2}{\int_{\Omega}|\phi|^2},\]
where the minimum is taken over all $k$-dimensional subspaces of $W^{1,2}_0(\Omega)$. Alternatively, there is an inductive description of eigenvalues (and eigenfunctions),
\[
\lambda_k(\Omega)=\min_{\phi}\frac{\int_{\Omega}|\nabla \phi|^2}{\int_{\Omega}|\phi|^2},\]
where the minimum is taken over all $\phi\in W^{1,2}_0(\Omega)$ which are orthogonal to the first $k-1$ eigenfunctions $\phi_{\lambda_1},...,\phi_{\lambda_{k-1}}$.

The analog of min-max formulas for the eigenvalues holds for eigenvalues of the Laplace-Beltrami operator on a compact manifold.\end{example}

\subsection{Courant nodal domain theorem} We denote by $Z(\phi)$ the zero set of a function $\phi$,
\[Z(\phi)=\{x: \phi(x)=0\}.\]
The connected components of the compliment $M\setminus Z(\phi)$ are called the nodal domains of the function $\phi$.
 
The simplest example of a compact manifold is the unit circle $\T\simeq[0,2\pi)$. Eigenfunctions of the Laplace operator are $2\pi$-periodic solutions of the eigenvalue problem
\[\phi''+\lambda \phi=0.\]
Solutions  exist when $\lambda=n^2$ for some integer $n$. For $n=0$ the first eigenfunction is a constant; for $n>0$ the corresponding eigenfunctions are linear combinations of $\phi_{n,1}(\theta)=\cos(n\theta)$ and $\phi_{n,2}(\theta)=\sin(n\theta)$. Each of them has $2n$ zeros on the circle. It is not difficult to see that this property is stable, if we change the metric on the circle the  eigenvalues $\lambda_n$ satisfy $\lambda_n\approx cn^2$ and the corresponding $2n$-th and $(2n+1)$st eigenfunctions have exactly $2n$ zeros, dividing the circle into $2n$ nodal intervals.  

The Courant nodal domain theorem gives an upper bound for the number of nodal domains of eigenfunctions on manifolds of arbitrary dimension. Let $M$ be a compact manifold as above and $\phi_{\lambda_n}$ be an eigenfunction of the Laplace-Beltrami operator corresponding to the $n$th smallest eigenvalue. 
\begin{theorem}[Courant]
The number of connected components of $M\setminus Z(\phi_{\lambda_n})$ is at most $n$.
\end{theorem}
We refer the reader to \cite[Chapter 6]{CH} and \cite{Ch} for proof and remark that the proof relies on the weak unique continuation property of solutions of second order elliptic PDEs, which in particular implies that an eigenfunction can not vanish on an open subset of a manifold. The aim of this notes is to give a new quantitative version of this uniqueness result.

\subsection{Further examples} First intuition on the geometry of zero sets of eigenfunctions comes from the pictures of nodal domains on the unit sphere and the standard torus, see \cite{Fig1, Fig2}.
\begin{example}\label{exm:sphere}
 The eigenfunctions on the unit sphere $S^d$ in $\R^{d+1}$ are restrictions of the homogeneous harmonic polynomials which are called spherical harmonics. If $P$ is a polynomial of $d+1$ variables,  $\Delta P=0$ and $P(x)=|x|^nY(x/|x|)$, where $Y$ is a function on $S=S^d$ then
\[ \Delta_SY+n(n+d-1)Y=0.\]
The spherical harmonics form a basis for $L^2(S^d)$ and there are no other eigenfunctions, further details are given in Exercise \ref{ex:sphere}.
\end{example}

\begin{example} Another standard compact manifold on which we can compute eigenfunctions explicitly is the torus. Let $\T^d$ be the $d$-dimensional torus, we identify it with a rectangle $\prod_{j=1}^d[-\pi,\pi ]$ which is glued along each pair of opposite sides. Then we have a basis of eigenfunctions of the form
\[\phi(x)=\exp\left(i\sum_{j=1}^d  n_jx_j\right),\quad \Delta_{\T^d} \phi+\sum_{j=1}^d n_j^2 \phi=0,\]
where $n_j\in\Z$.
\end{example}

We notice that in dimension $d>1$ there are eigenvalues for the Laplace--Beltrami operators on $S^d$ and $\T^d$ with arbitrary large multiplicities.
This  is a source of interesting examples of eigenfunctions. 

The zero sets of standard spherical harmonics and  eigenfunctions on the torus are not difficult to visualize, however, the structure of the zero sets of linear combinations of these functions (corresponding to the same eigenvalue) may be complicated.

\subsection{Bessel functions and Helmholtz equation} One more manifold that we  consider is $\R^d$. It is not compact. We consider bounded solutions of the Helmholtz equation 
\[\Delta \phi+\lambda \phi=0.\]
For $\lambda\le 0$ the maximum principle holds and there are no non-trivial bounded solutions. Thus we are interested in the case $\lambda>0$ and, rescaling the variable, we may assume that $\lambda=1$. 

The Laplace operator in polar coordinates can be written as
\[\Delta \phi=\partial_r^2\phi+\frac{d-1}{r}\partial_r \phi+\frac{1}{r^2}\Delta_S\phi.\] 
We look for solutions of the equation $\Delta\phi+\phi=0$ of the form $\phi(x)=f(|x|)Y(x/|x|)$. Separating the variables, one can check that $Y$ is an eigenfunction of the Laplace--Beltrami operator on the unit sphere. The eigenvalues on the sphere are given in Example \ref{exm:sphere} (see also Exercise \ref{ex:sphere} below). Then we find a family of solutions of the Helmholtz equation of the form
\[\phi(x)=f_n(|x|)Y\left(\frac{x}{|x|}\right),\quad \Delta_SY=-n(n+d-2)Y,\]
where $f_n(r)$ satisfies the following ordinary differential equation
\[r^2f''+(d-1)rf'+(r^2-n(n+d-2))f=0.\]
Writing $f_n(r)=r^{1-d/2}g_n(r)$ we see that $g_n(r)$ satisfies the Bessel equation
\[r^2g''+rg'+(r^2-(n+d/2-1)^2)g=0.\]
This is a second order ODE with analytic coefficients and it has a solution $J_{n+d/2-1}$ called the Bessel function (of the first kind) which is continuous at the origin. The solution is of  the form $J_{n+d/2-1}(r)=r^{n+d/2-1} h_{n+d/2-1}(r)$ where $h_{n+d/2-1}(r)$ is an analytic function of $r$ and $h_{n+d/2-1}(0)\neq0$ (see for example \cite{T}); the second solution has a singularity at $r=0$. Thus we get
\[f_n(r)=r^{1-d/2}J_{n+d/2-1}(r)=r^nh_{n+d/2-1}(r).\]
 We consider positive zeros of $J_{\nu}$ and enumerate them (one can check that they are simple)
\[0<j_{\nu,1}<j_{\nu,2}<... .\]
Using the obtained description of the solutions of the Helmholtz equation, we can compute eigenfunctions and eigenvalues of the Dirichlet Laplace operator for the unit ball in $\R^d$, see Exercise \ref{ex:bessel} below.


\subsection{Yau's conjecture} Examples of eigenfunctions on the torus and sphere show that the number of nodal domains may vary, it is bounded from above as we know from  the Courant nodal domain theorem. At the same time there exist eigenfunctions with large eigenvalues and just two nodal domains (it was noticed already in 1925 in the dissertation of Antonie Stern, see \cite{BH} for historical details and references).

On the other hand, these examples show that nodal lines become more complicated and dense when the eigenvalue grows. We give a proof of a well known result on the density of the zero sets of eigenfunctions in the next section. First we formulate a deep conjecture of Yau \cite{Y}.
\begin{conj}[Yau] Let $M$ be a smooth compact $d$-dimensional Riemannian manifold. There exist constants $C_1$ and $C_2$, which depend on $M$, such that 
\[C_1\sqrt{\lambda}\le \h^{d-1}(Z(\phi_\lambda))\le C_2\sqrt{\lambda},\]
for any eigenfunction $\phi_\lambda$ satisfying $\Delta_M \phi_\lambda+\lambda \phi_\lambda=0$.
\end{conj}
The nodal set of an eigenfunction is a union of smooth hypersurfaces with finite $(d-1)$-dimensional Hausdorff measure. The finiteness of the Hausdorff measure of the nodal set is a non-trivial fact, we refer the reader to \cite{HS} for details.

The Yau conjecture was proved for the case of real analytic metric by Donnelly and Fefferman in 1988, \cite{DF}. We outline some of the ideas in Section \ref{ss:AReal}.

\subsection{Lift of eigenfunctions}
The following lifting trick is used intensively in the study of eigenfunctions. Let $M$ be a $d$-dimensional manifold and $\phi_\lambda$ be an eigenfunction, $\Delta_M \phi_\lambda+\lambda \phi_\lambda=0$,
we define a function on $M'=M\times \R$ by
\[h(x,t)=\phi_\lambda(x)e^{\sqrt{\lambda}t}.\]
Then $\Delta_{M'}h=0$. Locally we think about $h$ as a solution of an elliptic equation in divergence form defined in a subdomain of $\R^{d+1}$.

The first application of the lifting trick is  the proof of the result on the density of the zero sets of eigenfunctions.
\begin{prop}
Suppose that $M$ is a compact Riemannian manifold. There exists $\rho=\rho(M)$ such that for any eigenfunction $\phi_\lambda$ with $\lambda>0$ and any $x\in M$ the distance from $x$ to the zero set $Z(\phi_\lambda)$ is less than $\rho\lambda^{-1/2}$. 
\end{prop}  
\begin{proof} Suppose that $\phi_\lambda$ does not change sign in some ball $B_r\subset M$. We assume that $r$ is small enough and consider a chart for $M$ that contains $B_r$. Then the function $h(x,t)=\phi_\lambda(x)exp(\sqrt{\lambda}t)$ is a solution of a  second order elliptic equation in divergence form and $h$  does not change sign in $B_r\times [-r,r]$. By the Harnack inequality, (see Theorem \ref{th:Harnack} below)
\[\sup_{D}|h|\le C(M)\inf_{D}|h|,\]
where $D=B_{r/2}\times[-r/2,r/2]$. On the other hand we have
\[\sup_D|h|=\sup_{B_{r/2}}|\phi_\lambda|\exp(r\sqrt{\lambda}/2)\ge \exp(r\sqrt{\lambda})\inf_{D}|h|.\]
It implies that $r<\rho\lambda^{-1/2}$.
\end{proof}
Clearly, if $h(x,t)=\phi_\lambda(x)\exp(\sqrt{\lambda}t)$ then the zero set of $h$ is the cylinder over $Z(\phi_\lambda)$ and the questions about $Z(\phi_\lambda)$ can be reformulated in terms of $Z(h)$. One of the advantages is that  $h$ is a solution of an elliptic second order PDE in divergence form with no lower order terms.


\subsection{A question of Nadirashvili}
Suppose that $h$ is a harmonic function in the unit disc $\D\subset\R^2$ such that $h(0)=0$. The zero set of $h$ is the union of analytic curves and by the maximum principle it has no loops. We assume that $h(0)=0$ then an elementary geometric argument implies that
\[\h^1(Z(h)\cap\D)\ge 2.\]
Nadirashvili asked whether a higher dimensional version of this statement holds.
\begin{conj}[Nadirashvili] There is a constant $c>0$ such that for any harmonic function $h$ in the unit ball $B$ of $\R^3$ such that $h(0)=0$, the following inequality holds
\[\h^2(Z(h)\cap B)\ge c.\]
\end{conj}
The question was formulated for harmonic functions and remained open for many years. The proof given recently in \cite{L2} is complicated (and beyond the scope of these lectures), it confirms the conjecture for solutions of  second order elliptic equation in divergence form with smooth coefficients.
\begin{theorem}[\cite{L2}]\label{th:L} Suppose that $Lu=\dv(A\nabla u)$ is a uniformly elliptic operator in the unit ball $B\subset\R^d$ with smooth coefficients. There exists a constant $c=c(A)$ such that for any solution of $Lu=0$ with $u(0)=0$ satisfies
\[\h^{d-1}(Z(u)\cap B)\ge c.\]
\end{theorem}

In was also shown in \cite{L2} that this theorem implies the lower bound in Yau's conjecture on compact Riemannian manifolds with smooth metric. A polynomial upper bound
\[\h^{d-1}(Z(\phi_\lambda))\le C\lambda^{A_d},\]
where $A_d$ depends only on the dimension of the manifold and $C$ depends on the manifold and the metric was obtained in  \cite{L1}. 


\subsection{Exercises}
\begin{exercise}[Harnack inequality]\ Let $L=div(A\nabla\cdot)$ be a uniformly elliptic operator with bounded coefficients.  Use the Harnack inequality (Theorem \ref{th:Harnack}) to prove the following statements.\\
a) If $u$ is a bounded solution of $Lu=0$ in $\R^d$ then $u$ is a constant.\\
b) Suppose that $Lu+cu=0$, $c\in\R$ and $u$ is positive in the cylinder 
\[{\mathcal{C}_1}=\{x=(x_1,...,x_d)\in\R^d: x_1^2+...+x_{d-1}^2\le 1\}.\]
Let further $M(R)=\max\{u(x): x\in{\mathcal C_{1/2}}, |x_d|\le R\}$,
where
\[{\mathcal{C}}_ {1/2}=\{x=(x_1,...,x_d)\in\R^d: x_1^2+...+x_{d-1}^2\le 1/4\}.\] Then there exists $C$ such that $M(R)\le u(0)e^{CR}$. 
\end{exercise}

\begin{exercise}\label{ex:first}  Suppose that $\Delta_M u+\lambda u=0$ and $\Omega$ is a connected component of $M\setminus Z(u)$. Assume that $\Omega$ is a domain with piece-wise smooth boundary and prove that the first Dirichlet Laplace eigenvalue of $\Omega$ is
\[\lambda_1(\Omega)=\lambda.\]
{\it{Remark:}} Careful details can be found in \cite{Che}.
\end{exercise}

\begin{exercise}[Harmonic polynomials]\label{ex:sphere} 
The restrictions of homogeneous harmonic polynomials on the unit sphere $S\subset\R^{d+1}$, called spherical harmonics,  are  the eigenfunctions of  the Laplace--Beltrami operator. We denote by $E_{n,d}$ the eigenspace that corresponds to the eigenvalue $\lambda=n(n+d-1)$. If $Y\in E_{n,d}$ then $P(x)=|x|^nY(x/|x|)$ is a harmonic function.\\
a) Apply the Green formula in $\R^d$ to show that if $Y_n\in E_{n,d}$ and $Y_m\in E_{m,d}$ with $n\neq m$ then
\[\int_S Y_nY_m=0.\]
b) Consider the following inner product on the space $\p_{n,d}$ of homogeneous polynomials of degree $n$,
\[[P,Q]=P(D)(Q)=\sum_{|\alpha|=n}\alpha ! P_\alpha Q_\alpha,\]
where $P(x)=\sum_{|\alpha|=n} P_\alpha x^\alpha,\ Q(x)=\sum_{|\alpha|=n} Q_\alpha x^\alpha.$

Show that the space of harmonic polynomials $\h_{n,d}\subset \p_{n,d}$ is the orthogonal compliment of \[\Q_{n,d}=\{P\in \p_{n,d}: P(x)=|x|^2P_1(x), P_1\in \p_{n-2,d}\}\] with respect to this inner product.
\\
c) Show that any homogeneous polynomial $F$  of degree $n$ in $\R^d$ can be written as
\[F(x)=H_n(x)+|x|^2H_{n-2}(x)+...|x|^{2k} H_{n-2k},\]
where $k=[n/2]$ and $H_j$ is a homogeneous harmonic polynomial of degree $j$.\\
{\it{Remark:}} This implies that spherical harmonics form a basis for $L^2(S)$ and there no other eigenfunctions.\\
d) Deduce that if $Y\in \h_{n,d}$ and $F$ is a polynomial of degree less that $n$ than $\int_S YF=0$.\\
e) Suppose that  $P(x)\in\h_{n,d}$ and   $Q$ is a factor of $P$, $P=QF$ for some polynomial $F$. Show that $Q$ changes sign in $\R^d$.
\end{exercise}

\begin{exercise}[Dirichlet eigenfunctions for balls]\label{ex:bessel} \ \\
Let $J_{n}$ be the Bessel function such that 
\[u(re^{i\theta})=J_{n}(r)(a\cos n\theta+b\sin n\theta)\] satisfies 
$\Delta u+u=0\ {\text{in}}\ \R^2,$
i.e., $J_{n}$ is a solution of the second order ODE
\[r^2J''+rJ+(r^2-n^2)J=0.\]
Further, let $0<j_{n,1}<j_{n,2}<...$ be the positive zeros of $J_{n}$.\\
a) Show that there is a constant $c$  such that  $n\le j_{n,1}\le cn$. (Hint: you may use the equation for the lower bound and the density of zero sets of eigenfunctions for the upper bound.)\\
b) Show that the following functions 
\[\phi(re^{i\theta})=J_{n}(j_{n,k}r)(a\cos n\theta+b\sin n\theta)\]
are eigenfunctions of the Dirichlet Laplacian on the unit ball of $\R^2$, and that  the smallest eigenvalue is $j^2_{0,1}$.\\
{\it{Remark 1:}} A classical and deep result of Siegel implies that two distinct Bessel functions $J_n$ and $J_m$ with integer $n$ and $m$ have no common zeros and thus all eigenvalues of a disk are simple.\\
{\it{Remark 2:}} Let $\lambda_{d,k}$ be the $k$th  eigenvalue of the Dirichlet Laplace operator on the unit ball $B_0\subset\R^d$. Suppose that $M$ is a smooth $d$-dimensional Riemannian manifold, $x\in M$ and let $B=B(x,r)$ be the ball on $M$ of radius $r$ and center $x$. Let $\lambda_k(B)$ be the $k$th eigenvalue of the Dirichlet Laplace-Beltrami operator for $B$. Then one can show that (see \cite{Ch})
$\lambda_k(B)\sim r^{-2}\lambda_{d,k},\quad r\to 0.$
\end{exercise}

\begin{exercise}[Yau's conjecture] \ \\
Prove the lower bound $\h^{d-1}(Z(u))\ge c\sqrt{\lambda}$ in the Yau conjecture in dimensions one and two. \\
{\it{Hint:}} for the case $d=2$ use Exercise \ref{ex:first}, the inequality 
$\lambda_1(\Omega_1)\ge \lambda_1(\Omega_2)$ for $\Omega_1\subset\Omega_2,$ and Remark 2 above.
\end{exercise}


\section{Doubling index and frequency function}\label{s:di}
An important tool to study nodal sets of eigenfunctions and growth properties of solutions of elliptic PDEs is the so called frequency function. The idea goes back to works of Almgren \cite{Al} and  Agmon \cite{Ag}. It was developed further by Garofalo and Lin \cite{GL}, see also \cite{K} and \cite{Man}. 


\subsection{Frequency function}
Let $A(x)$ be a symmetric uniformly elliptic matrix with Lipschitz coefficients defined on some ball $B_r$ centered at the origin and such that $A(0)=I$. Let further
\[\mu(x)=\frac{(A(x)x,x)}{|x|^2},\quad \mu(0)=1,\quad \Lambda^{-1}\le \mu(x)\le\Lambda.\]
Moreover, since $A$ has Lipschitz coefficients, we have $A(x)=I+O(|x|)$ and $\mu(x)=1+O(|x|)$.

Let $u$ be a solution to the equation $\dv(A(x)\nabla u(x))=0$. We consider weighted averages of $|u|^2$ over spheres:
\[H(r)=r^{1-d}\int_{\partial B_r}\mu(x)|u(x)|^2ds(x).\]
Denoting by $\nu=x/|x|$ the unit outer normal vector for the sphere and applying  the divergence theorem, we obtain
\[H(r)=r^{-d}\int_{\partial B_r}(|u|^2A(x)x,\nu)ds=r^{-d}\int_{B_r}\dv(|u|^2A(x)x).\]
For the case of the Laplace operator, $A=I$ and $\mu(x)=1$, the function $t\mapsto H(e^t)$ is convex, i.e.,
\[H(r)\le H(r_1)^\alpha H(r_2)^{1-\alpha},\quad {\text{when}}\quad r=r_1^\alpha r_2^{1-\alpha},\ \alpha\in(0,1).\] 
This can be proved either using decomposition of harmonic functions in spherical harmonics or by integration by parts as below, the computations are slightly simplified in this case, see \cite{H}. Similar property was discovered for solution of elliptic equations in \cite{GL}, we will provide a calculation that is a small variation of the one in \cite{K}.

First we compute the derivative of $H$,
\begin{equation}\label{eq:H1}
H'(r)=-dr^{-1}H(r)+r^{-d}\int_{\partial B_r}\dv(|u|^2A(x)x).
\end{equation}
We rewrite the second term as
\begin{multline*}\int_{\partial B_r}\dv(|u|^2A(x)x)=\\
\int_{\partial B_r}2u(\nabla u, A(x)x)+\int_{\partial B_r}|u|^2{\rm{trace}}(A(x))+\int_{\partial B_r}|u|^2A_D(x),
\end{multline*}
where $A_D(x)=\sum_{i,j}(\partial_i a_{ij})x_j$. We also note that 
\[\mu(x)=1+O(|x|),\ {\rm{trace}}(A)=d+O(|x|),\ A_D(x)=O(|x|).\]
This implies
\begin{equation}\label{eq:div}
\int_{\partial B_r}\dv(|u|^2A(x)x)=\int_{\partial B_r}2u(\nabla u, A(x)x)+d\int_{\partial B_r}|u|^2\mu(x)+O(r^dH(r)).
\end{equation}
We rewrite the first integral in the right-hand side of the last identity using the symmetry of $A$ and then apply the divergence theorem once again,
\[\int_{\partial B_r}2u(\nabla u, A(x)x)=\int_{\partial B_r}2u(A(x)\nabla u, x)=
2r\int_{B_r}\dv(uA(x)\nabla u).\]
Next, using the equation $\dv(A\nabla u)=0$, we obtain
\begin{equation}\label{eq:dvA}
 \int_{\partial B_r}2u(\nabla u, A(x)x)=2r\int_{B_r} (A(x)\nabla u,\nabla u).
\end{equation}
Finally, combining (\ref{eq:H1}), (\ref{eq:div}), and (\ref{eq:dvA}), we get
\[H'(r)=2r^{1-d}\int_{B_r}(A\nabla u,\nabla u) +O(H(r)).\]
Following \cite{GL} and \cite{K}, we define
\[I(r)=r^{1-d}\int_{B_r}(A\nabla u, \nabla u)=r^{-d}\int_{\partial B_r}(uA\nabla u,x),\quad N(r)=\frac{rI(r)}{H(r)}.\]
Then 
\begin{equation}\label{eq:H'}
H'(r)=2I(r)+O(H(r)),\quad N(r)=\frac{rH'}{2H}+O(1).
\end{equation}

\begin{prop}\label{pr:mon}
There exists $C$  that depends only on the ellipticity and Lipschitz constants of the operator such that
for any solution $u$ to $\dv(A\nabla u)=0$, the function $e^{Cr}N(r)$ is an increasing function of $r$. 
\end{prop}

\begin{proof}
We compute $N'(r)$, taking into account that the first derivatives of the coefficients of $A$ are bounded. We already know that
\[H'(r)=2I(r)+O(H(r)).\]
Next we estimate $(rI(r))'$. Let $w$ be a vector field in $B_r$ such that $(w,x)=r^2$ on $\partial B_r$. Then
\begin{align}(rI(r))'&=(2-d)I(r)+r^{2-d}\int_{\partial B_r}(A\nabla u,\nabla u)\nonumber\\
&=(2-d)I(r)+r^{1-d}\int_{B_r}\dv(w(A\nabla u,\nabla u))\label{eq:I1}\\&= (2-d)I(r)+r^{1-d}\int_{B_r}\dv(w)(A\nabla u,\nabla u)+r^{1-d}\int_{B_r}(w,\nabla(A\nabla u,\nabla u))\nonumber.\end{align}
We used the divergence theorem in the first equality above.
To simplify the last term we note that
\begin{equation}\label{eq:In}
(w,\nabla(A\nabla u,\nabla u))=2(w,{\rm{Hess}}(u)(A\nabla u))+(A_{D,w}\nabla u,\nabla u),
\end{equation}
where $A_{D,w}(x)=\{\sum_{k}(\partial_k a_{ij})w_k\}_{i,j}$.
Further, the Hessian is a symmetric matrix and 
\[{\rm{Hess}}(u)(w)=\nabla(\nabla u,w)-(Dw)\nabla u.\]
Thus,  we obtain,
\begin{align}
\int_{B_r}(Hess(u)w,A\nabla u)&=
\int_{B_r}(\nabla(\nabla u,w),A\nabla u)-\int_{B_r}((Dw)\nabla u,A\nabla u)\nonumber\\
&=\int_{B_r}\dv((\nabla u,w)A\nabla u)-\int_{B_r}((Dw)\nabla u,A\nabla u)\label{eq:I2}
\\&=r^{-1}\int_{\partial B_r}(\nabla u, w)(A\nabla u,x)-\int_{B_r}((Dw)\nabla u,A\nabla u).\nonumber
\end{align}
We used the equation in the second identity and the divergence theorem in the third.

Now we choose $w(x)=\mu(x)^{-1}A(x)x$. Then \[(w(x),x)=|x|^2,\quad Dw=I+O(|x|),\quad \dv (w)=d+O(|x|).\]
We proceed to work with  (\ref{eq:I2}) and rewrite the first term as
\[\int_{\partial B_r}(\nabla u,w)(A\nabla u,x)=\int_{\partial B_r}\mu(x)^{-1}(A\nabla u,x)^2.\]
Combining the second term in \eqref{eq:I1} and the second term in \eqref{eq:I2} and the asymptotic for $Dw$ and $\dv(w)$, we get
\[r^{1-d}\int_{B_r}\dv(w)(A\nabla u, \nabla u)-2r^{1-d}\int_{B_r}((Dw)\nabla u, A\nabla u)=(d-2)I(r)+O(rI(r)).\]
Moreover, we have
\[r^{1-d}\int_{B_r}|(A_{D,w}\nabla u,\nabla u)|\le C r^{1-d}\int_{B_r} r|\nabla u|^2=O(rI(r)),\]
where $C$ depends on the ellipticity and Lipschitz constants of $A$ and on the dimension.
Now \eqref{eq:I1}, \eqref{eq:In}, \eqref{eq:I2} and the last two inequalities imply
\[
(rI(r))'=2r^{-d}\int_{\partial B_r}\mu(x)^{-1}(A\nabla u,x)^2+O(rI(r)).
\]
Finally, the last inequality and \eqref{eq:H'} give
\begin{multline*}
N'(r)(N(r))^{-1}=(rI(r))'(rI(r))^{-1}-(H'(r))(H(r))^{-1}\\
=\frac{2r^{-2d}}{I(r)H(r)}\left(\int_{\partial B_r}\frac{(A\nabla u,x)^2}{\mu(x)}\int_{\partial B_r}\mu(x)|u|^2-\left(\int_{\partial B_r}(uA\nabla u, x)\right)^2\right)+O(1).
\end{multline*}
The first term is positive by the Cauchy-Schwarz inequality. Therefore \[N'(r)\ge -CN(r)\]  and the proposition follows.
\end{proof}

\begin{corollary}\label{cor:Double} Suppose that $u$ is a solution to the equation $\dv(A(x)\nabla u(x))=0$ in $B_{R_0}$, where $A(x)=I+O(x)$ as above and $R<R_0/2$. Then there exists $D_N$ that depends on $R_0$, $N(R)$, the ellipticity and Lipschitz constants of the operator, and  the dimension of the space, such that
\[\int_{B_{2r}}|u|^2\le D_N\int_{B_r}|u|^2\]
for any $r\in(0,R)$.
\end{corollary}
\begin{proof} For any $r<R<R_0/2$ we write \eqref{eq:H'} and apply the proposition
\[H'(r)\le 2I(r)+cH(r)=(2r^{-1}N(r)+c)H(r)\le (2r^{-1}N(R)e^{C(R-r)}+c)H(r).\]
Integrating $H'(r)/H(r)$ over an interval $[\rho, 2\rho]$ we get
\[\int_{\partial B_{2\rho}}\mu(x)|u(x)|^2ds(x)\le C_N \int_{\partial B_{\rho}}\mu(x)|u(x)|^2ds(x),\]
where $C_N=\exp(C_1+C_2N(R))$ with $C_2=C_2(R)$. Finally, integrating the inequality with respect to $\rho$ from $0$ to $r$, and using that  $\Lambda^{-1}\le\mu\le \Lambda$ we obtain
the required estimate.
\end{proof}


\subsection{Three spheres theorem for elliptic PDEs}
Another consequence of the monotonicity of the frequency function is the so called three sphere theorem. Its simplest version is the classical Hadamard three circle theorem for analytic functions (the classical proof is based on the fact that the logarithm of the modulus of an analytic function is subharmonic). It turns out that even without analyticity a version of the Hadamard inequality holds for harmonic functions and more generally solutions to uniformly elliptic equations. One of the first general results is due to Landis \cite{L63}.
We derive the three spheres from  the properties of  the frequency function following \cite{GL}.  
First,  Proposition \ref{pr:mon} implies the inequality $e^{Cr}N(2r)\ge N(r)$, which, combined with \eqref{eq:H'}, gives
\[\frac{rH'(r)}{H(r)}\le \left(c+\frac{2r H'(2r)}{H(2r)}\right)e^{Cr}.\]
Then integrating from $r$ to $2r$ with respect to $dr/r$ we obtain
\begin{equation}\label{eq:log}
\log H(2r)-\log H(r)\le (c\log 2+\log H(4r)-\log H(2r))e^{2Cr}.
\end{equation}

\begin{prop}\label{pr:3sph}
Assume that $L=\dv(A\nabla\cdot)$ is a uniformly elliptic operator, $A$ is symmetric and has Lipschitz entries in a domain $\Omega$. Suppose also that  $A(0)=I$ and $B(0,4r)\subset\Omega$. There exist   $\alpha>0$ and  $C>0$ such that for any solution $u$ of $Lu=0$ 
\[\int_{\partial B_{2r}} |u|^2\le C\left(\int_{\partial B_r}|u|^2\right)^{\alpha}\left(\int_{\partial B_{4r}}|u|^2\right)^{1-\alpha}.\]
\end{prop}
\begin{proof} We collect similar terms in \eqref{eq:log} and take the exponent of both sides to obtain
\[
\int_{\partial B_{2r}}\mu|u|^2ds\le C_1\left(\int_{\partial B_r}\mu|u|^2ds\right)^{\alpha}\left(\int_{\partial B_{4r}}\mu|u|^2ds\right)^{1-\alpha}\]
with $\alpha=(1+e^{4Cr})^{-1}$ and $\alpha$ can be chosen close to $1/2$ as $r\to 0$. 
This  inequality and bounds on $\mu$ imply the required estimate. 
\end{proof}

Assume that $A$ is as above with $A(0)=I$.
The corollary above and  the equivalence of norms (see Corollary \ref{cor:norms} below) imply the following three ball inequality for supremum norms 
\[ \sup_{B_{2r}}|u|\le C \left(\sup_{B_r}|u|\right)^{\alpha_1}\left(\sup_{B_{8r}}|u|\right)^{1-\alpha_1},\]
for some $C$ and $\alpha_1\in(0,1)$.
Then by local change of variables we can drop the assumption that $A(0)=I$, balls are replaced by ellipses. Applying the inequality several times and inscribing ellipses in balls we obtain the following statement. (We omit some technical details required for an accurate argument.) 
\begin{corollary}\label{cor:3ball1}
There exist $r_0>0$, $k$ large enough, $C$ and $\beta\in(0,1)$ such that if $B=B_r$ is a ball with $r<r_0$ and $B_{kr}\subset\Omega$ then
\[\sup_{B_{2r}}|u|\le C\left(\sup_{B_r}|u|\right)^\beta\left(\sup_{B_{kr}}|u|\right)^{1-\beta}.\]
\end{corollary}
The general version of this result can be obtain by the chain argument.
\begin{corollary}\label{cor:3ball}
 Let $B\subset K\Subset \Omega$, where $B$ is open and $K$ is compact. There exist $C$ and $\gamma\in(0,1)$ that depend only on $K,\Omega, B$ and the ellipticity and Lipschitz constants of $L$ such that for any solution $u$ to $Lu=0$ in $\Omega$ the following inequality holds
\[\sup_{K}|u|\le C\left(\sup_B|u|\right)^\gamma\left(\sup_\Omega|u|\right)^{1-\gamma}.\]
\end{corollary}
\begin{proof}[Proof: Chain argument]
Assume that $\sup_{\Omega}|u|=1$. For each point $x\in K$ there is a curve $\gamma$ connecting $x$ to some fixed point in $B$. We then can find a finite sequence of balls $\{B_j\}_{j=1}^J$ such that $r(B_j)<r_0$, $B_1\subset B$, $B_{j+1}\subset 2B_j$, $kB_j\subset \Omega$ and  $x\in B_J=B(x)$. Applying the previous corollary we see that
\[\sup_{B_{j+1}}|u|\le \sup_{2B_j}|u|\le C(\sup_{B_j}|u|)^\beta.\]
Iterating this estimate we obtain
\[\sup_{B_J}|u|\le C_1(\sup_B|u|)^{\beta_1}.\]
Finally, we take a finite cover of $K$ by balls $B(x)$ and get the required estimate.
\end{proof}


\subsection{Doubling index}
We prefer to replace the frequency function by a comparable but more intuitive quantity that we call the doubling index.
Let $h\in C(\Omega)$, such that $h$ does not vanish on any open subset of $\Omega$. For any ball $B$ such that $2B\subset\Omega$ we define
\[\n_h(B)=\log\frac{\max_{2B}|h|}{\max_B|h|}.\]
Note that if  $p$ is a homogeneous polynomial of degree $n$ and $B$ is centered at the origin than $\n_p(B)=n\log 2$. At the same time if we compute the frequency function $N_P(r)$ defined for the case of the Laplace operator ($A=I$), we get
$N_p(r)=2n$. In general, if $h$ is a solution to $Lh=0$ then using the equivalence of norms (Corollary \ref{cor:norms}) and the estimate in the proof of Corollary \ref{cor:Double}, we obtain that for $r<r_0$
\[C_1^{-1}N_h(r)-C_2\le \n_h(B_r)\le C_1N_h(4r)+C_2.\]
 The inequality above and the almost monotonicity of the frequency implies the following almost monotonicity for the doubling index
\begin{equation}\label{eq:mondouble}
\n_h(B_r)\le C(\n_h(B_R)+1) 
\end{equation}
when $4r<R<R_0$.

\begin{remark}It is well known that the doubling index is connected to the size of the zero set of harmonic function at least in dimension two, we refer to works of Gelfond \cite{G}, Robinson \cite{R}, Nadirashvili \cite{N1} and an exposition in \cite{NPS}. Let $h$ be a harmonic function in the unit disk $\D$ and let $\beta(r)=\#\{Z(h)\cap r\T\}$ be the number of zeros of $h$ on the circle of radius $r$.
Assume that $h(0)=0$ then the following inequalities hold
\[c_1\beta(r)\le \n_h(2r\D)\le c_2\beta(4r).\]
\end{remark}


\subsection{Doubling index for eigenfunctions} The monotonicity of the doubling index and three sphere theorem hold for solutions of second order elliptic equations of the form $\dv(A\nabla h)=0$. For eigenfunctions $\phi_\lambda(x)$ on compact manifolds there is no monotonicity of the doubling index and the three sphere inequality gets a constant that depends on the eigenvalue. As above, we consider the lift $h(x,t)=e^{\sqrt{\lambda}t}\phi_\lambda(x)$ and then we can apply the results of the previous sections to $h$ that solves an equation of the form $\dv(A\nabla h)=0$. 

Donnelly and Fefferman used the doubling index in their study of nodal sets of eigenfunctions on smooth manifolds. One of their basic results for general smooth compact Riemannian manifolds is the following.
\begin{prop} Let $M$ be a smooth compact Riemannian manifold. There exists $r_0$ and $C$ depending on $M$ such that for any eigenfunction $\phi=\phi_\lambda$, 
\[\Delta_M\phi_\lambda+\lambda \phi_\lambda=0\]
 the doubling index $\n_{\phi}(B)\le C\sqrt{\lambda}$ when $B$ is a ball on $M$ with radius $r\le r_0$.
\end{prop}

\begin{proof}
We note that $\n_{\phi}(B)\le \n_{h}(B')+C\sqrt{\lambda}$, where $B'$ is a ball in $M\times \R$ with center on $M\times{0}$ and the same radius as $B$, we say that $B'$ is the lift of $B$. It is enough to prove the estimate for the doubling index of $h$ on $M\times[-R,R]$. Assume that $\max_M|\phi|=|\phi(x_0)|=1$ and fix $r$ such that  for each point $x\in M$ the geodesic ball $B_r(x)$ is contained in a chart.  

Let $B$ be any ball of radius $r/2k$ on $M$ and $B'$ be its lift  in $M\times\R$.
We consider a finite chain $\{B_j\}_{j=1}^J$ of geodesic balls in $M\times[-r,r]$ with centers on $M\times 0$ and equal radii $r/2k$. We choose the balls such that $B_1=B'$,  $B_{j+1}\subset 2B_{j}$ and $(x_0,0)\in B_J$. Then by Corollary \ref{cor:3ball1}
\[\sup_{B_{j}}|h|\ge c(\sup_{2B_{j}}|h|)^{1/\beta}\ge c(\sup_{B_{j+1}}|h|)^{1/\beta}.\]
It implies that $\sup_{B'}|h|\ge c_1$, where $c_1=c_1(r)$, and then
\[e^{r\sqrt{\lambda}}\sup_{B}|\phi|\ge c_1,\]
where $c_1$ depends on $r$ and $M$ (which also determine the number of balls in a chain). Thus for any ball $B$ of radius $r$ (ot larger)  and the corresponding lifted ball $B'$ we obtain $\n_\phi(B)\le C(\sqrt{\lambda}+1)$ and $\n_h(B')\le C(\sqrt{\lambda}+1)$.  Finally, the almost monotonicity of the doubling index for $h$ implies similar estimate for balls of radius less than $r$.

\end{proof}


\subsection{Cubes}
In the next sections  a version of the doubling index for cubes will be useful. For a given cube  $Q\subset \R^d$ we denote its side length by $s(Q)$. Then the volume of the cube is
$|Q|=(s(Q))^d$.

 Assume that $u$ is a solution to the equation $Lu=0$ in a domain $\Omega\subset\R^d$ and for each cube $Q$ with $2Q\subset\Omega$  define
\begin{equation}\label{eq:double-def}
\n_u(Q)=\sup_{q\subset Q}\log\frac{\max_{2q}|u|}{\max_{q}|u|}.
\end{equation}
We claim that the almost monotonicity of the usual doubling index implies
that the supremum above is finite. By the definition, we have now that if $q\subset Q$ then  $\n_u(q)\le \n_u(Q)$. 

We want to compare $\n_u(Q)$ to $\log\max_{2Q}|u|-\log\max_Q|u|$. Take a cube $q\subset Q$. If $q$ is small, $s(q)<c_d s(Q)$, we first apply almost monotonicity inequality for the doubling index \eqref{eq:mondouble}.  Let $b$ be the largest ball inscribed in $q$ then $2q\subset k_db$, where $k_d=2\sqrt{d}$ and we have 
\[\log\frac{\max_{2q}|u|}{\max_q|u|}\le \log\frac{\max_{k_db}|u|}{\max_b|u|}\le C_1\log\frac{\max_{k_dB}|u|}{\max_{B}|u|}+C_2,\]
where $B$ is a ball concentric with $b$ such that $k_dB\subset Q$, $R=R(B)\sim s(Q)$. This implies
\[\frac{\max_{2q}|u|}{\max_q|u|}\le C_3\left(\frac{\max_{k_dB}|u|}{\max_{B}|u|}\right)^{C_1}.\]
Now, using that $R(B)$ is comparable to $s(Q)$,  we repeat the chain argument from the proof of Corollary \ref{cor:3ball} to obtain the inequality
\[\max_Q|u|\le C\left(\max_{B}|u|\right)^\gamma\left(\max_{2Q}|u|\right)^{1-\gamma}.\]
with $C$ and $\alpha\in(0,1)$ which does not depend on $B$ (for $B$ with $R(B)\sim s(Q)$ the number of balls in the chain is uniformly bounded). 
Finally,
\[\frac{\max_Q|u|}{\max_{2Q}|u|}\le C\left(\frac{\max_{B}|u|}{\max_{k_dB}|u|}\right)^{\gamma}\le C\left(\frac{\max_{q}|u|}{\max_{2q}|u|}\right)^{\gamma/C_1}.\]
For large cubes $q$ with $s(q)\ge c_d s(Q)$ the last inequality follows directly from the three balls inequality and the chain argument. Thus we obtain
\begin{equation}\label{eq:invdouble}
\log\frac{\max_{2Q}|u|}{\max_Q|u|}\ge a_1\n_u(Q)-a_2,
\end{equation}
where $a_1$ and $a_2$ depend on the ellipticity and Lipschitz constants of the operator only when we assume that $s(Q)\le 1$.
 
We also consider eigenfunctions on manifolds and define the doubling index for eigenfunctions over cubes  in a similar way, to prove that the supremum is finite for this case we can use the monotonicity for the lifted function.


\subsection{Remarks on the size of the zero sets of eigenfunctions and the doubling index}\label{ss:AReal}
 In this section we first formulate some results that were proved by Donnely and Fefferman \cite{DF}. We assume that $M$ is a Riemannian manifold the metric is  real analytic (or that coefficients of the corresponding elliptic operator are real-analytic).
\begin{lemma} Let $L=\dv(A\nabla\cdot)$ be a uniformly elliptic operator in the unit cube $Q_0\subset\R^{d+1}$ with real analytic coefficients. There is  constant $C=C(L)$ such that if $Lh=0$ and $\n_{h}(2Q_1)\le N,\ N>1$ then
\[\h^{d}(Z(h)\cap Q_1)\le CNs(Q_1)^{d}.\]
 \end{lemma}
We don't know if this lemma remains true for non-analytic case. Suppose that $\phi_\lambda$ is an eigenfunction on a compact manifold $M$ with real-analytic metric. Applying this lemma to $h(x,t)=\phi_\lambda(x)\exp(\sqrt{\lambda}t)$ on charts and having in mind the bound for the doubling index of $h$, one obtains the upper bound in Yau's conjecture
\[\h^{d-1}(Z(\phi_\lambda))\le C\sqrt{\lambda}.\]

This part of the conjecture is open for non-analytic manifolds. The best known result, see \cite{L1}, is based on a non-analytic version of the lemma above, the estimate is
\[\h^d(Z(h)\cap Q_1)\le CN^As(Q_1)^{d}\]
for some $A=A(d)$. It implies  a polynomial bound in Yau's conjecture.

To obtain the lower bound in the Yau's conjecture on manifolds with real analytic metric, Donnelly and Fefferman proved the following statement.
\begin{lemma} Suppose that $M$ is a real-analytic manifold. There exists $N_0$ such that the following is true. If $\phi=\phi_\lambda$ is an eigenfunction on $M$ and  $M$ is partitioned into cubes with side length $\approx \sqrt{\lambda}^{-1}$, $M=\cup q$, then for at least half of these cubes $\n_\phi(q)\le N_0$.
\end{lemma}
This lemma can be combined with the next one (applied for the lifted function) to give the conjectured lower bound for the size of the zero set of eigenfunctions on real-analytic manifolds.

\begin{lemma} \label{l:low} Let $L=\dv(A\nabla\cdot)$ be a uniformly elliptic operator with smooth coefficients in the unit cube $Q_0\subset\R^{d+1}$. There exists a function $f(N)$ that depends only on $L$ such that if $Lh=0$ in $Q_0$, $h(0)=0$ and $\n_h(Q_1)\le N$ then 
\[\h^{d}(Z(h)\cap Q_1)\ge f(N)s(Q_1)^d.\]
\end{lemma}
 The last lemma does not require analyticity of the coefficients. A simple quantification of this estimate is known (see  remarks in \cite{LM1}); the statement of \ref{l:low} is weaker than Theorem \ref{th:L}. 

We conclude this lecture by  formulating an estimate for the size of the zero set from above which is not as precise as  the polynomial bound in \cite{L1}. It follows from earlier results by Hardt and Simon \cite{HS}. 

\begin{lemma}\label{l:up} Let $L=\dv(A\nabla\cdot)$ be a uniformly elliptic operator with smooth coefficients in the unit cube $Q_0\subset\R^{d+1}$. There exists a function $F(N)$ that depends only on  $L$ such that if $Lh=0$ in $Q_0$, and $\n_h(Q_1)\le N$ then 
\[\h^{d}(Z(h)\cap Q_1)\le F(N)s(Q_1)^d.\]
\end{lemma}


\subsection{Exercises}
\begin{exercise} Let $h$ be a harmonic function on $\R^d$. The frequency function of $h$ is defined by 
\[N(r)=\frac{rH'(r)}{2H(r)},\]
where $H(r)=r^{1-d}\int_{|x|=r}|h(x)|^2ds(x).$\\
a) Show that if $h$ is a homogeneous polynomial of degree $n$ then $N(r)=n$.\\
b) Let $h=\sum_{k=l}^L p_k$, where $p_k$ is a homogeneous harmonic polynomial of degree $k$ and $p_l, p_L\neq 0$. Show that 
\[\lim_{r\to 0}N(r)=l\quad{\text{and}}\quad \lim_{r\to\infty}N(r)=L.\]
Remark: $l$ is called the vanishing order of $h$ at the origin.\\
c) Use the fact that $N(r)$ is a non-decreasing function to prove that
\[\left(\frac{R}{r}\right)^{2N(r)}\le \frac{H(R)}{H(r)}\le \left(\frac{R}{r}\right)^{2N(R)}.\]
\end{exercise}

\begin{exercise}[Applications of the three ball inequality]\ \\
 Suppose that $h$ is a non-constant harmonic function in  $\R^d$ such that $|h|\le 1 $ on a half-space $\{x=(x_1,...,x_d, x_d>0\}$. Let
\[m(R)=\max_{|x|<R}|h|.\]
a) Show that there exist $c>0$ and $\alpha\in(0,1)$ such that 
\[m(R)\le C m(5R)^\alpha.\]
b) Show that $m(R)\ge c\exp(R^\beta)$ for some $\beta>0$.
\end{exercise}

\begin{exercise}[Log-convex functions]\ \\
Let $m:\R_+\to\R_+$ be a continuous function. We say that $m$ is log-convex if 
$f(t)=\ln(m(\exp(t)))$ is a convex function. (For example if $m(x)=x^a, a>0$ then $f(t)=at$ and $m$ is log-convex.) Warning: usually a positive function $g$ is called logarithmically convex if $log (g)$ is a convex function.\\
a) Show that if $a_k$ are non-negative numbers then $m(x)=\sum_{k=1}^n a_k x^k$ is log-convex.\\
{\it{Remark:}} It is true that the sum of two log-convex functions is log-convex.\\
b) Let $u$ be a harmonic function in the unit ball of $\R^d$, we know that
\[u(x)=\sum_{k=0}^\infty |x|^kY_k(x/|x|),\]
where $Y_k$ is an eigenfunction of the Laplace-Beltrami operator on the unit sphere $S\subset\R^d$.
Show that
\[m(r)=\int_{S}|u(ry)|^2ds(y)\]
is log-convex.\\
c) Let  $K(x,t)$ be the heat kernel in $\R^d$, 
\[K(x,t)=(4\pi t)^{-d/2}\exp(-|x|^2/(4t)),\]
and it satisfies the equation $\Delta K(x,t)=\partial_t K(x,t)$. Suppose that $u$ is a harmonic function in $\R^d$ such that $u(x)\exp(-c|x|^2)\in L^2(\R^d)$ for any $c>0$. Define
\[ M(t)=\int_{\R^d} |u(x)|^2K(x,t)dt.\]
Compute $M'(t)$ and show that $M^{(m)}(t)\ge 0$ for any $m$.\\
 {\it{Remark:}} The positivity of all derivatives implies that $M(t)$ is a log-convex function. This convexity was studied by Lippner and Mangoubi \cite{LiMa} for the case of discrete harmonic functions. 
\end{exercise}

\begin{exercise}[Reverse H\"older inequality for solutions of elliptic equations]\label{ex:revH}
Show that if $u$ is a solution of a uniformly elliptic equation with Lipschitz coefficients, $\dv(A\nabla u)=0$ in a ball $B_0$  then for some (any) $q>1$ there exists $C_q(u)$ such that  for any ball $B\subset 1/2B_0$
\[\left(|B|^{-1}\int_B|u|^{2q}\right)^{1/q}\le C_q(u)|B|^{-1}\int_B|u|^2.\]
{\it{Remark:}} It implies that $|u|^2$ is a  Muckenhoupt weight and therefore  $|Z(u)|=0$. Similar inequality holds for function $u-|B|^{-1}\int_B u$ and together with Caccioppoli inequality it implies that $|\nabla u|^2$ is also a Muckenhoupt weight (see \cite{GL} for details).
\end{exercise}


\section{Small values of polynomials and solutions of elliptic PDEs}\label{s:R}
We start with a non-constant polynomial $P\in\C[z]$ of one complex variable with complex coefficients,
\[P(z)=a_nz^n+a_{n-a}z^{n-1}+...+a_1z+a_0.\]
 As $|z|$ grows the behavior of $P(z)$ resembles that of the highest degree term $a_nz^n$. As we know $P(z)$ has $n$ zeros counting multiplicities and   the set $\{z: |P(z)|<C\}$ is bounded and contains the zeros. We use the notation
\[E_a(P)=\{z: |P(z)|<e^{-a}\}.\]

\subsection{Classical results of Cartan and Polya} A classical result on the set where a polynomial takes small values is due to H. Cartan. We denote by $\p_n$ the set of all  polynomials of degree $n$ with leading coefficient $1$,
\[\p_n=\{p(z)=z^n+a_{n-1}z^{n-1}+...+a_1z+a_0\in\C[z]\}.\]
\begin{lemma}[Cartan, 1928]
Let $p\in\p_n$ then for any $a,\alpha>0$ there exist a finite collection of balls $\{B_j\}$ such that $E_{na}(p)\subset\cup_j B_j$ and $\sum_jr_j^\alpha\le e(2e^{-a})^{\alpha}$, where $r_j$ is the radius of $B_j$
\end{lemma}
In particular, taking  $\alpha=2$ one obtains that $|E_{na}(p)|\le 4\pi e^{1-2a}$. This estimate is not sharp as the next result shows.
\begin{lemma}[Polya, 1928]
Let $p\in\p_n$ then $|E_{na}(p)|\le \pi e^{-2a}$ for any $a>0$. 
\end{lemma}
The last inequality is sharp, the equality is obtained when $p(z)=z^n$.

Lemmas of Cartan and Polya deal with polynomials for which the leading coefficient is equal to one and provide estimates of the set of all points of the complex plane where the polynomial is small, the proofs of both lemmas and related results can be found in \cite{Lub}. We are interested in a local version of such estimate. 


\subsection{Remez' inequality for polynomials}
Now we consider polynomials with real coefficients on the real line and do not normalize the leading coefficient.

\begin{lemma}[Remez, 1936]
Let $E$ be a measurable subset of an interval $I$ of positive measure,  $|E|>0$. Then for any polynomial $P_n\in \R[x]$ of degree $n$ 
\[\max_{x\in I}|P_n(x)|\le \left(\frac{4|I|}{|E|}\right)^n\max_{x\in E}|P_n(x)|\]
\end{lemma}
More precise inequality and its proof is outlined in the exercises below, see Exercise \ref{ex:Rzin}. We reformulate the inequality in the following way
\[|E|\le 4|I|\left(\frac{\max_{x\in E}|P_n(x)|}{\max_{x\in I}|P_n(x)|}\right)^{1/n},\]
when $E\subset I$. 
We normalize $P_n$ such that $\max_I|P_n|=1$ and use the notation
\[E_{an}(P_n)=\{x\in \R: |P_n(x)|<e^{-an}\}.\]
Then the Remez inequality can be written as
\[|E_{an}(P_n)\cap I|\le 4|I|e^{-a}.\]
There are interesting generalizations of the Remez inequality, in particular the measure of the set can be replaced by another geometric characteristic; higher dimensional version are also known, we refer the reader to \cite{E, BY}. 


\subsection{Propagation of smallness result}
The main result we prove in these lectures is the following version of quantitative propagation of smallness for solutions of elliptic equation in divergence form. As above we assume that $\dv(A\nabla\cdot)$ is a uniformly elliptic operator, $A$ is a symmetric matrix with Lipschitz coefficients on some domain in $\R^d$. We know that a solution to $\dv(A\nabla h)=0$ can not vanish on a set of positive measure (see for example Remark after Exercise \ref{ex:revH}) and look for a quantitative version of this result.

\begin{theorem}[\cite{LM}]\label{th:Main}
 Let $h$ be a solution of
$\dv(A\nabla h)=0$ in $\Omega$.
 Assume that
\[|h|\le \varepsilon\quad {\text{ on}}\quad E\subset\Omega,\]
where $|E|>0$. Let  $K$ be a compact subset of $\Omega$ then
\begin{equation}\label{eq:ellRem}
\max_K|h|\le C_0\sup_{\Omega}|h|^{1-\alpha} \varepsilon^{\alpha},
\end{equation}
where $C_0>0$ and $\alpha\in(0,1)$ depend on $A, |E|, \dist(E,\partial\Omega)$, and $K$.
\end{theorem}
The inequality \eqref{eq:ellRem} can be considered as a version of three balls theorem where the smallest ball is replaced by a measurable set. The constants in the inequality depend on the measure of the set and the distance from this set to the boundary of $\Omega$ but not on the set itself, which could be an arbitrarily wild measurable set. The question whether such inequality holds was asked by Landis, weaker quantitative estimates were obtained by Nadirashvili \cite{N} and Vessella \cite{V}.

We formulate first the following result (Remez inequality for solutions of elliptic PDE, \cite{LM}):\\
{\it{ Let $Q$ be the unit cube in $\R^d$. Assume $h$ is a solution to the equation $\dv(A\nabla h)=0$ in $2Q$ and define the doubling index $N=\n_h(Q)$ as in \eqref{eq:double-def}.
 Then for any subset $E$ of $Q$ of positive Lebesgue measure
\begin{equation}\label{eq:RL}
\sup_{Q}|h|\le C \sup_E|h|\left(C\frac{|Q|}{|E|}\right)^{CN}
\end{equation}
where $C$ depends on $A$ only.}}

This statement confirms that in some sense solutions of elliptic equations locally behave as polynomials with degree bounded by the multiple of the doubling index. In particular (the lift of) an eigenfunction corresponding to eigenvalue $\lambda$ behaves as a polynomial of degree $C\sqrt{\lambda}$. This was pointed out in the works of Donnelly and Fefferman, see for example \cite{DF2}, where in particular an interesting Bernstein type inequality for eigenfunctions is obtained.

Let us show that \eqref{eq:RL} implies Theorem \ref{th:Main}. First we remind that by \eqref{eq:invdouble}
\[\exp(a_1N)\le e^{a_2}\sup_{2Q}|h|(\sup_{Q}|h|)^{-1},\]
for some $a_1,a_2>0$.
Suppose that  \eqref{eq:RL} holds with some constant $C$ and  choose $C_1=C_1(|E|)$ such that
\[\left(C\frac{|Q|}{|E|}\right)^C=e^{a_1C_1},\quad \text{i.e.}\ C_1=Ca_1^{-1}\log(C|Q||E|^{-1}).\] 
Then
\[
\sup_Q|h|\le C\sup_E|h|\exp(a_1C_1N)\le C_2\sup_E|h|\left(\sup_{2Q}|h|\right)^{C_1}\left(\sup_Q|h|\right)^{-C_1}.\]
This implies the inequality in the theorem for the case $\Omega=2Q$ and $K=Q$ with $\alpha=(C_1+1)^{-1}$ and $C_0$ that depends on $|E|$ and on $A$ but not on $h$. To obtain the statement of the theorem we use the standard chain argument as in the proof of Corollary \ref{cor:3ball}.

In its turn the inequality \eqref{eq:RL} is equivalent to the following local estimate of the volume of sub-level sets. We recall that $Q$ is the unit cube.
\begin{lemma} \label{l:R} Suppose that $\dv(A\nabla h)=0$ in $2Q$ and $\sup_{Q}|h|=1$. Let  further $N =\n_h(Q) \geq 1$ and
\[E_a(h)=\{x\in Q: |h(x)|<e^{-a}\}.\] 
Then 
\begin{equation}\label{eq:RPDE}
|E_a(h)|\le Ce^{-\beta a/N}|Q|,
\end{equation}
for some positive $C$ and $\beta$ that depend on $A$ only. 
\end{lemma}


\subsection{Base of induction}\label{ss:base}
 We prove Lemma \ref{l:R} in the next section using double induction in $a$ and $N$. Now  we check the base of induction, considering two cases  $a\le c_0N$ and $N\le N_0$. 

Our aim is to prove the inequality (\ref{eq:RPDE}). First we note that for $a/N<c_0$ the inequality holds trivially, indeed if we choose the constant $C=C(\beta)$ large enough, we get \[Ce^{-\beta a/N}\ge Ce^{-\beta c_0}\ge 1.\]

Now we want to show that (\ref{eq:RPDE}) holds for some $\beta$ and $C$ if we assume that $N$ is small enough.
 The lemma below  is the base of our induction in $N$.

\begin{lemma}\label{l:base}  Assume that $h$ satisfies $\dv(A\nabla h)=0$ in $k_dQ$,  $\sup_{Q}|h|=1$ and $\n_h(Q) \leq N_0$. Let 
$E_a=\{x\in Q: |h(x)|<e^{-a}\}.$ 
Then 
\begin{equation}\label{eq:baseN}
|E_a|\le C e^{-\gamma a}|Q|,
\end{equation}
for some $\gamma=\gamma(N_0, A)$ and $ C =C (N_0, A)$.
\end{lemma}

The estimate on the doubling index implies that $\sup_{1/2Q}|h|\ge C(N_0)$. We combine this inequality with the  oscillation theorem (see Theorem \ref{th:osc} in Appendix). Recall that $\osc_{Q}h=\sup_Qh-\inf_Qh.$
\begin{theorem}
Let $L=\dv(A\nabla\cdot)$ be a uniformly elliptic operator in $\Omega$ and $Lh=0$. There exists $\tau=\tau(s)<1$  depending on $s$ an the ellipticity constant such that for any cube $Q\subset\Omega$ 
\[\osc_{sQ}h<\tau(s)\osc_{Q} h\quad {\text{and}}\ \tau(s)\to 0\ {\text{as}}\ s\to 0.\]
\end{theorem}

\begin{corollary} Assume that $h$ satisfies $\dv(A\nabla h)=0$ in $2Q$,  $\sup_{Q}|h|=1$ and $\n_h(Q) \leq N_0$. There exist an integer $K$, and positive $b$ and $m$ that depend on $N_0$, on the ellipticity constants of $A$ and on the dimension $d$  
such that if $Q$ is partitioned into $K^d$ smaller equal cubes, $Q=\cup q$ then
\[\sup_q|h|\ge b\quad{\text{for any}}\ q\]
and there exists one cube $q_0$ in the partition such that $\inf_{q_0} |h|>m$.
\end{corollary}

\begin{proof} Since $\n_h(Q)\le N_0$, we get a lower bound on the supremum of $|h|$ on each small cube $q$. 
Further, assume that  $h(x_0)=\max_{Q/2}|h|\ge c(N_0)$ and $K$ is chosen large enough (we replace $h$ by $-h$ if necessary). We take $q_0$ such that $x_0\in q_0$. Clearly $\osc_Qh\le 2$ and  since $K/2q_0\subset Q$ by the oscillation theorem we have $\osc_{q_0} h\le 2\tau(2/K)$. Then we conclude 
\[\inf_{q_0}h= \sup_{q_0}h-\underset{q_0}\osc\, h\ge c(N_0)-2\tau(2K^{-1})\ge  m,\]
when $m<c(N_0)/2$ and $K$ is large enough.
\end{proof}

In particular, the corollary implies that $|\{x\in Q: |h|<m\}|\le (1-K^{-d})|Q|.$
Dividing each $q$ once again  into smaller cubes, we get on each new cube the supremum of $|h|$ is at least $b^2$ and  
\[|\{x\in Q:|h|<mb\}|\le (1-K^{-d})^2|Q|.\]
Iterating the corollary we see that
\[|\{x\in Q: |h|<mb^l\}|\le (1-K^{-d})^{l+1}|Q|,\]
when $\sup_{Q}|h|=1$. Thus the estimate \eqref{eq:baseN} holds for $e^{-a}=b^lm$ and $\gamma$ such that $b^\gamma=1-K^{-d}$, it completes the proof of the Lemma \ref{l:base}.


\subsection{Exercises}
\begin{exercise} Let $f\in L^2(\T^2)$, $\|f\|_{L^2}=1$. We define the $L^2$-doubling index of $f$ on a square $q$ by
\[n(f,q)=\log\frac{\int_{2q}|f|^2}{\int_q|f|^2}.\]
Assume that $\T^2$ is partitioned into $K^2$ equal squares we say that a square is good if $n(f,q)<100$.
Show that 
\[\sum_{q\ {\text{good}}}\int_{q}|f|^2\ge 1/2.\]
{\it{Remark:}} $1/2$ is a very rough estimate, you can find a better one.
\end{exercise}

\begin{exercise}[Discrete version of Remez inequality] Use the Remez inequality to show that if $P$ is a polynomial of degree $n$ and  $S\subset I\cap\Z$ contains $n+m$ points then
\[\max_I|P|\le \left(\frac{4|I|}{m}\right)^n\max_S|P|.\]
\end{exercise}

\begin{exercise}[Remez inequality for polynomials]\label{ex:Rzin}
Let $T_n(x)$ be the Chebyshev polynomial or degree $n$, such that $T_n(\cos \theta)=\cos (n\theta)$. This sequence can be defined by
\[T_0(x)=1,\ T_1(x)=x,\ T_{n+1}(x)=2xT_n(x)-T_{n-1}(x).\]
Clearly for each $n$ there is a sequence $-1= x_{n,0}<x_{n,1}<...<x_{n,n}=1$ such that $T_n(x_k)=(-1)^{n-k}$. 

 Suppose that $c>0$ and $E\subset I=[-1,1+c]$ is a measurable set with $|E|=2$. In this exercise we prove that for any polynomial $P$ of degree $n$
\[\max_{I}|P|\le T_n(1+c)\max_E|P|.\]
the equality is obtained for example when $E=[-1,1]$ and $P=T_n$.
To prove the inequality it is enough to assume that $E$ is open and show that \[P(1+c)\le T_n(1+c)\max_E|P|.\] 
a)  Show that there are points $y_k, k=0,...,n$ in $E$ such that $|x_{n,k}-x_{n,j}|\le |y_k-y_j|$ and $1+c-x_{n,k}\ge 1+c-y_k$.\\
b) Use the Lagrange interpolation formula and the properties of the Chebyshev polynomials to show that $P(1+c)\le T_n(1+c)\max_E|P|.$ \\
c) Let $x>1$, show that $T_n(2x-1)\le (4x)^n$. \\
{\it{Remark:}} This gives a proof of the Remez inequality formulated in the lecture notes.
\end{exercise}

\begin{exercise} [Quantitative unique continuation for harmonic functions]
We will use Remez inequality to show the quantitative unique continuation form sets of positive measure for harmonic functions.\\
a) Suppose that $h$ is a bounded harmonic function in the unit ball $B_0$. Let $r<r_0(d)$ be small enough. Show that there exists $q(r)<1$ and $C$ such that for any integer $n$ there is a polynomial $p_n$, $\deg p_n\le n$  such that
\[\max_{|x|\le r}|h(x)-p_n(x)|\le Cq(r)^n\max_{|x|\le 1}|h(x)|.\]
Moreover $q(r)\to 0$ as $r\to 0$.\\
b) Prove that there is $r_1=r_1(d)$ such that if $E$ is a measurable subset of $r_1B_0$ of positive measure, $m=|E|$, and $h$ is a harmonic function in $B_0$ then
\[\max_{r_1B_0}|h|\le C(\max_E|h|)^\alpha(\max_{B_0}|h|)^{1-\alpha},\]
where $\alpha$ depends on $m$  and $r_1$.
\end{exercise}

\begin{exercise}[logarithmic capacity]\ \\
Define the logarithmic capacity of a compact subset of the complex plane by
\[{\rm{cap}}(K)=\lim_{n\to\infty}\left(\min_{p\in\p_n}\max_K|p(x)|\right)^{1/n}.\]
a) Show that the limit exists.\\
b) Prove that  ${\rm{cap}}(E_{na}(p))=e^{-a}$ for  any $p\in\p_n$.\\
c) Use Polya's lemma to show that $|K|\le \pi{\rm{cap}}(K)^2$ for any compact set $K\subset\C$.\\
\end{exercise}


\section{Proof of propagation of smallness result}\label{s:proof}
We now prove Lemma \ref{l:R} using double induction on $a$ and $N$ and some iterative argument. First we prove some preliminary result on the distribution of the doubling index that will help us to carry on the induction step.

\subsection{On distribution of the doubling indices}
The results on  the doubling index that we formulate below are crucial for the proof. Let $Q_0$ be the unit cube in $\R^d$.

First we assume that $f\in C(Q_0)$ and for any $q$ such that $2q\subset Q_0$ we define
\[N_f(q)=\log\frac{\max_{2q}|f|}{\max_q|f|}.\]
Warning: We have used the notation $N_h(r)$ for the frequency of $h$ in the ball $B(0,r)$ in Section \ref{s:di}. But for the rest of the notes we do not refer to the frequency function and use $N_f(q)$ for the doubling constant of $f$ in a cube $q$ as defined above.  
\begin{lemma} \label{l:LM1}
 Let a cube $Q\subset Q_0$ be partitioned into $K^d$ equal cubes $q_i$, $K\ge 8$. Put $N_{\min}= \min\limits_{i} N_f(q_i)$ and assume that $N_{\min}$ is large enough, $N_{min} \geq N_0(d)$. Then 
\[N_f(Q/2) \geq  \frac{K}{8} N_{\min}.\]
\end{lemma}

\begin{proof}
Let $\max_{Q/2}|f|=|f(x_0)|$, $x_0\in q_i$ for some $i$. Since $N_f(q_i)\le N_{min}$ there exists $x_1\in 2q_i$ such that $|f(x_1)|\ge e^{N_{min}}|f(x_0)|$. Clearly $x_1\in (1/2+2/K)Q$. We can find one of the cubes in the partition for which $x_1\in q$ and repeat the step. Then there is a sequence of points $x_j$ such that \[|f(x_j)|\ge e^{jN_{min}}|f(x_0)|\] and $x_j\in (1/2+2j/K)Q$. We repeat this $\left[K/4\right]$ times, the last $x$ is in $Q$. Then 
\[\max_{Q}|f|\ge e^{KN_{min}/8}\max_{Q/2}|f|.\]
which implies the required estimate.
\end{proof}

For solutions of elliptic equations we can formulate the above result using the monotonicity of the doubling index and the modified quantity  $\n_h(q)$.

\begin{corollary} \label{cor:cubes} Let $L=\dv(A\nabla \cdot)$ be a uniformly elliptic operator in $2Q_0$. There exist constants $N_0$ and $J_0$ such that if
$Lh=0$ in $2Q_0$, $Q\subset Q_0$, $Q$ is partitioned into $J^d$ cubes $q_i$ with $J\ge J_0$ then for at least one $q$
\[\n_h(q)\le\n_h(Q)/2.\]
\end{corollary}

We rewrite  the inequality \eqref{eq:invdouble}in the following way
\[N_h(q)\le\n_h(q)\le A_1N_h(q)+A_2.\]
Then Corollary follows immediately from Lemma \ref{l:LM1}.

Our aim in induction argument is to divide the cube  into  small cubes and find a subcube with small doubling index.


\subsection{Choosing the right notation}
We fix the ellipticity constant $\Lambda>1$ and the Lipschitz constant  $C$  and consider second order elliptic operator $L=\dv(A\nabla \cdot)$ in the cube $2Q_0$, where $Q_0$ is the unit cube  in $\R^d$. We vary the parameters $N > 1$ and $a>0$ and aim at proving the estimate (\ref{eq:RPDE}).
 
Let \[m(u,a)=|\{x\in Q_0: |u(x)|<e^{-a}\sup_{Q_0}|u|\}|\] and
 \[M(N,a)=\sup_{*}m(u,a),\] where the supremum is taken over all elliptic operators $\dv(A \nabla \cdot )$ and functions $u$ satisfying the following conditions in $2Q_0$:
\begin{itemize}
 \item[(i)]  $A(x)=[a_{ij}(x)]_{1\le i,j\le d}$ is a symmetric uniformly elliptic matrix with Lipschitz entries and ellipticity and Lipschitz constants bounded by $\Lambda$ and $C$ respectively,
 \item[(ii)] $u$ is a solution to $\dv(A\nabla u)=0$ in $2Q_0$,
\item[(iii)] $\n_u(Q_0)\le N$. 
\end{itemize}
Our aim is to show that
\begin{equation}\label{eq:pro}
M(N,a)\le Ce^{-\beta a / N}.
\end{equation}
 The constant $\beta>0$ will be chosen later and will not depend on $N$.

 As we remarked in Section \ref{ss:base} we can assume that $a/N>c_0$. By Lemma \ref{l:base} we can also assume that $N$ is sufficiently large.
The proof  now contains two main steps.
 First, with the help of Corollary \ref{cor:cubes} we  prove a recursive inequality for $M(N,a)$. Then we show that the recursive inequality implies the exponential bound  \eqref{eq:RPDE} by a double induction argument on $a,N$. 


\subsubsection{Recursive inequality.} We show that for some $a_0>0$ and $s<1$
\begin{equation} \label{eq:rec}
 M(N,a) \leq   M( N/2, a - Na_0 )+ s M(N, a -Na_0).
\end{equation}

 Fix  a solution $u$ to the elliptic equation $\dv(A \nabla u)=0$  with $\n_u(Q_0)\le N$. Divide $Q_0$ into $J^d$ subcubes $q$ and apply Corollary \ref{cor:cubes}. It claims that at least one cube $q_0$ satisfies $\n_u(q_0)  \leq N/2$. We have
\begin{equation*} 
m(u,a)=\sum\limits_{q}|\{x\in q: |u(x)|<e^{-a}\sup_{Q_0}|u|\}|.
\end{equation*}
 By  the definition of the doubling constant we see that
\begin{equation*}
\sup_{q}|u|\ge c_1 J^{-C_1 N}\sup_{Q_0}|u|.
\end{equation*} 
 Since $N$ is sufficiently large, we can forget about $c_1$ above by increasing $C_1$ and we have
\begin{equation*}
\sup_{q}|u|\ge e^{-a_0N}\sup_{Q_0}|u|.
\end{equation*} 
 We continue to estimate $m(u,a)$:
\begin{multline*} 
m(u,a)\le  \sum_{q}|\{x\in q: |u(x)|<e^{-a+a_0N}\sup_{q}|u|\}|\\
= |\{x\in q_0: |u(x)|<e^{- a+a_0N} \sup_{q_0}|u|\}|+ \sum_{q\neq q_0} |\{x\in q: |u(x)|<e^{- a+a_0N} \sup_{q}|u|\}|.
\end{multline*}
 Now, we estimate the first term,
\begin{equation*}
|\{x\in q_0: |u(x)|<e^{- a+a_0N} \sup_{q_0}|u|\}| \leq J^{-d} M(N/2, a-a_0N).
\end{equation*}
 We use the  fact that the restriction of $u$ to the cube $2q$ corresponds to a solution of another elliptic PDE, the new equation can be written in the divergence form with some coefficient matrix  which has the same bounds for ellipticity and Lipschitz constants.

For the second term we get
\begin{equation*}
\sum_{q\neq q_0} \leq (J^d-1)J^{-d} M(N, a-a_0N) = s M(N,  a-a_0N),
 \end{equation*}
where $s=(J^d-1)J^{-d}<1$.
 Adding the inequalities for the first and second terms and taking the supremum over $u$, we obtain the recursive inequality \eqref{eq:rec} for $M(N,a)$.


\subsection{Recursive inequality implies  exponential bound.}
We will now prove that 
\begin{equation}\label{eq:pro2}
M(N,a)\le Ce^{-\beta a / N}
\end{equation}
for some $C$ large enough and $\beta>0$ small enough by a double induction on $N$ and $a$.
 Without loss of generality we may assume $N=2^l$, where $l$ is an integer number. 
 Suppose that we know \eqref{eq:pro2} for $N=2^{l-1}$ and all $a>0$ and now we wish to establish it 
 for $N=2^{l}$.  By Lemma  \ref{l:base} we may assume 
 $l$ is sufficiently large. 
 For a fixed $l$ we argue by induction on $a$ with step  $a_02^l$. 
 We may assume that  $a/N > k_0a_0$, where $k_0>0$ will be chosen later.
  For $a\leq k_0a_0N$ the inequality is true if we choose the constant $C$ large enough. 
The induction base implies the inequality for $k=k_0$. We describe the step of the induction from $a=(k-1)a_02^l$ to $a=ka_02^l$.

 By the induction assumption we have 
\[M(2^l,(k-1)a_02^l)\le Ce^{-\beta (k-1)a_0}\]
and 
\[ M(2^{l-1},(k-1)a_02^l)\le Ce^{- 2\beta (k-1)a_0}.\]
We apply the recursive inequality \eqref{eq:rec} 
\[M(2^l,ka_02^l) \leq Ce^{-2\beta (k-1)a_0}+Cse^{-\beta(k-1)a_0}. \]
 Our goal is to obtain the following inequality 
\[ e^{-2\beta(k-1)a_0}+se^{-\beta (k-1)a_0} \leq e^{-\beta ka_0}\]
for $k>k_0$ and some $\beta>0$.
Dividing by $e^{-ka_0\beta}$ we reduce it to
\[e^{-\beta a_0(k-2)}+ se^{\beta a_0}\le 1.\]
The last inequality holds with the proper choice of the parameters: $s<1$ and  $a_0$ are fixed,  we choose $\beta$ to be small enough so that the second term is less than $(1+s)/2$ and then choose large $k_0$ to make the first term smaller than $(1-s)/2$ when $k\ge k_0$. This concludes the induction step and the proof of our main result.

More delicate propagation of smallness from sets of codimension smaller then one is discussed in \cite{LM}.


\subsection{Excercesis}
\begin{exercise} Suppose that $Lu=0$ in the unit cube $Q_0$. \\
a) Use the oscillation theorem to show that there exists a constant $K$ which depends on the Lipschitz and ellipticity constants for $L$ such that if $q$ is a small cube with $Kq\subset Q$ and $Z(u)\cap q\neq\emptyset$ then
\[\log\frac{\max_{Kq}|u|}{\max_{q}|u|}\ge 2.\]
b) Show that there exists $c$ and $B_0$ such that if $Q_0$ is partitioned into $B^d$ cubes $q$, $B>B_0$ and $Z(u)\cap q\neq \emptyset$ for each $q$ then
\[N_u(Q/2)=\log\frac{\max_{Q}|u|}{\max_{Q/2}|u|}\ge cB,\]
where $c$ depends on $K$ from a). 
\end{exercise}

\begin{exercise}
Assume that $m:\Z_+\times\Z_+\to \R_+$ satisfies
\[m(k,j)\le C\ {\text{for}}\ j< 4,\quad m(1,j)\le e^{-j},\] 
\[{\text{and}}\ m(k,j)\le m(k-1,2(j-1))+\frac{1}{4}m(k,j-1).\]
Prove that $m(k,j)\le Ce^{-j}$. \\
{\it{Remark:}} A similar argument is used to derive the estimate in the lecture notes from the iterative inequality.
\end{exercise}

\begin{exercise}[Remez inequality for eigenfunctions]\ \\
a) Let $M$ be a compact manifold. Use the lift and the Remez inequality for solutions of elliptic equations to show that there exists a constant $C=C(M)$ such that for any eigenfunction $\phi_\lambda$ and any compact set $E\subset M$ the following inequality holds.
\[\max_E|\phi_\lambda|\ge C^{-1}\max_M|\phi_\lambda|\left(\frac{|E|}{C|M|}\right)^{C\sqrt{\lambda}}.\]
b) Let $M=S^2$ and $B$ be a small ball on $S^2$, construct a sequence of eigenfunctions $\phi_\lambda$ on the sphere with $\lambda\to \infty$ such that $\sup_B|\phi_\lambda|/\sup_M|\phi_\lambda|$ decays as $e^{-c\sqrt{\lambda}}$. 
\end{exercise}

\begin{exercise}
 Apply the Remez inequality for solutions of elliptic equations   to show that if $h$ is a solution of  $Lh=0$ in $kQ_0$ then $g=\log|h|$ is in BMO and $\|g\|_{BMO(Q_0)}\le C_L\n_h(Q_0)$.\\
Reminder: A function $g$ is said to have bounded mean oscillation if there exists a constant $C$ such that 
\[\frac{1}{|Q|}\int_Q|g-c_Q|\le C\]
for any cube $Q$ and some constants $c_Q$. The smallest $C$ for which the inequality holds is called the $BMO$-norm of $g$.\\
In particular if a function $g$ satisfies
\[|\{x\in Q: |g(x)-c_Q|>\gamma\}|\le C\exp(-A\gamma)|Q|,\]
for some $c_Q$ then $g\in BMO$ and $\|g\|_{BMO}\le c/A$.
\end{exercise}

\section{Appendix:  Second order elliptic equations in divergence form}\label{s:PDE}


\subsection{Basic results for elliptic operator in divergence
 form}
We  study solutions of second order elliptic equations in divergence form
\[Lu:=\dv(A\nabla u)+cu=0,\]
where $u\in W^{1,2}(\Omega)$, i.e., $|\nabla u|\in L^2(\Omega)$, $\Omega\subset\R^d$. The  matrix $A=A(x)$ is symmetric and uniformly elliptic, i.e., 
\[\Lambda^{-1}|v|^2\le(A(x)v,v)\le \Lambda|v|^2\]
for any $x\in\Omega$ and any $v\in \R^d$. 

First we assume that the elements of $A(x)$ are measurable bounded functions (the boundedness follows from the uniform ellipticity condition). We will assume that  $c$ is measurable and bounded, weaker integrability assumptions on $c$ are sufficient for some of the results below.
The equation $Lu=0$ is understood in the integral sense, similarly, we consider the inequalities $Lu\ge 0$ and $Lu\le 0$.
The first classical result is the maximal principle, see for example \cite[Theorem 8.1]{GT}. We use here the standard notation, $u^+=\max(u,0)$. 
\begin{theorem}[Maximal principle] Suppose that $c\le 0$ and $u\in W^{1,2}(\Omega)$ satisfies $Lu\ge 0$. Then
\[\sup_{\Omega} u\le \sup_{\partial\Omega} u^+.\]
\end{theorem}

We also use the following classical inequality for gradients of solutions of general elliptic PDEs in divergence form. 
\begin{theorem}[Caccioppoli inequality] Suppose that $Lu=0$ in $\Omega$, $B_R\subset\Omega, r<R$. Then
\[\int_{B_r}|\nabla u|^2\le C\left(\frac{1}{(R-r)^2}+\|c\|_{L^\infty}\right)\int_{B_R}|u|^2,\]
where $C=C(d,\Lambda)$.
\end{theorem}

Classical iteration methods of De Giorgi and Moser imply the following estimates, see \cite[Chapter 4]{HL}
\begin{theorem}[Local boundedness]
Suppose that $Lu\ge 0$ in $\Omega$, $2B\subset\Omega$, then $u^+\in L^\infty_{loc}(\Omega)$ and 
\[\sup_{B}u^+\le C\left(|2B|^{-1}\int_{2B}|u^+|^2\right)^{1/2},\]
where $C$ depends on $d,\Lambda$ and $\|c\|_\infty$.
\end{theorem}
This gives immediately the equivalence of norms
\begin{corollary}\label{cor:norms}
Suppose that $Lu=0$ in $2B_0$, where $B_0$ is the unit ball of $\R^d$, then 
\[C_1\|u\|_{L^2(B)}\le\|u\|_{L^\infty(B)}\le C_2\|u\|_{L^2(2B)},\]
where $C$ depends on $d,\Lambda$ and $\|c\|_\infty$.
\end{corollary}

Another part of the regularity theory that goes back to De Giorgi and Moser is the following oscillation theorem (see \cite[Chapter 4]{HL}).

\begin{theorem}\label{th:osc}[Oscillation inequality] Let $L=\dv(A\nabla\cdot)$ be a uniformly elliptic operator in $\Omega$. There exists $q=q(\Lambda)<1$ such that for any ball $B$ such that $2B\subset\Omega$ 
\[\sup_{B}u-\inf_Bu<q(\sup_{2B}u-\inf_{2B}u).\]
\end{theorem}
The difference  $\sup_Bu-\inf_Bu$ is called the oscillation of the function $u$ in $B$ and denoted by $\osc_Bu$.

A different way to obtain regularity  was discovered by Landis (see \cite{L63} for details) and developed to elliptic equations is non-divergence form with bounded coefficients by Krylov and Safonov, see
\cite{L63, Lbook, KS, S}. This approach also leads to the oscillation inequality.

Finally, we formulate the Harnack inequality of Moser for solutions of elliptic equations in divergence form, see for example \cite[Chapter 4]{HL}.
\begin{theorem}[Harnack inequality]\label{th:Harnack} Let $u$ be a non-negative solution to  elliptic equation $\dv(A\nabla u)=0$ in $\Omega$, $2B\subset\Omega$. Then
\[\sup_{B}u\le C\inf_{B} u,\quad C=C(d,\Lambda).\]
\end{theorem}
There is a nice proof of the Harnack inequality for solutions of elliptic equations in divergence form that bypasses the classical iteration methods can be found in \cite{S}. Note that in all of the results in this section the constants depend on the ellipticity constant only, thus we may apply the inequalities on small or big scales.


\subsection{Comparison to harmonic functions}
We will turn to elliptic PDEs in divergence form with Lipschitz coefficients. This smoothness assumption allows us to freeze the coefficients and consider the equation as a perturbation of the equation with constant coefficients. Changing coordinates, we can think about constant coefficient elliptic operator as a simple transformation of the usual Laplace operator. More precisely, let $u$ be a solution to
\[\dv(A\nabla u)=0,\]
where $A=\{a_{ij}(x)\},\ x\in\Omega$ and
\[|a_{ij}(x)-a_{ij}(y)|\le C|x-y|.\]
Then for any $x_0\in\Omega$ there is a ball $B_r(x_0)$ and a linear transformation $S:B_\rho(0)\to B_r(x_0)$ such that $f=u\circ S$ is a solution of elliptic equation $\dv(\tilde{A}\nabla f)=0$ with
\[\tilde{A}(0)=I,\quad |\tilde{a}_{ij}(y)-\delta_{ij}|\le C|y|.\]
Moreover $r/\rho$ is bounded, the bound depends on the ellipticity and Lipschitz constants for $A$.

We  mostly study local properties of solutions and then  reduce the problem to equation of this specific form. Note that when we apply this idea we get inequalities that hold on small scales, the constants depend on the Lipschitz constants of the coefficients and may grow when we consider large balls.

Classical regularity result implies that if $u\in W^{1,2}(\Omega)$ is a weak solution of the divergence form elliptic equation as above (with Lipschitz coefficients) and $\Omega'\Subset\Omega$ then $u\in W^{2,2}(\Omega')$ and then if $\partial\Omega'$ is smooth then by the trace property $u, |\nabla u|\in L^2(\partial\Omega')$.

 \subsection*{Acknowledgments} These notes are based on the lectures given by the second author at Park City Mathematics Institute Summer Program in July 2018. It is a great pleasure to thank the organizers for this great opportunity and wonderful time in Park City. We are grateful to many students and colleagues who attended the lectures and commented on earlier versions on the manuscript. In particular to Paata Ivanisvili, who held problem sessions for these lectures, and to Stine Marie Berge, who carefully read and commented the first version of the lecture notes.

%
%
%
%
%
%
%


\end{document}